\documentclass[
11pt]
{amsart}

\usepackage[pagebackref]{hyperref}
\usepackage[margin=1.2in, marginpar=0.5in, marginparsep=.1in]{geometry}
\usepackage{hyperref}
\usepackage{cite}
\usepackage{amssymb}
\usepackage{amsmath}
\usepackage{amsthm}
\usepackage{amsfonts}
\usepackage{graphicx}%
\usepackage{comment}
\usepackage{bbm}
\usepackage{xcolor}
\usepackage[toc,page]{appendix}
\usepackage{mathtools}
\usepackage{amsmath}

\numberwithin{equation}{section}

\theoremstyle{plain} 
\newtheorem{thm}{Theorem}[section] 
\newtheorem{assumption}[thm]{Assumption}
 
\newtheorem{lem}[thm]{Lemma} 
\newtheorem{prop}[thm]{Proposition} 
\newtheorem{notation}[thm]{Notation} 
\newtheorem{rmk}[thm]{Remark}

\newcommand{\R}{\mathbb{R}}  
\newcommand{\E}{\mathbb{E}} 
\newcommand{\Prob}{\mathbb{P}}

\title[Chung's LIL and small balls for spherical random fields]{A note on small probabilities for spherical random fields at a critical regime}

\author[Carfagnini]{Marco Carfagnini}
\address{School of Mathematics and Statistics\\
University of Melbourne \\
Parkville VIC 3010, Australia}
\email{marco.carfagnini@unimelb.edu.au}

\keywords{Time-dependent spherical random fields, spherical harmonics, small ball probabilities, Chung's law of the iterated logarithm.}

\subjclass{60G60, 60G17, 60G15, 42C05. }

\begin{document}

	\begin{abstract}
    We consider time-dependent space isotropic and time stationary spherical Gaussian random fields. We establish  Chung's law of the iterated logarithm  and solve the  small probabilities problem. Our results depend on the high-frequency behaviour of the angular power spectrum: the speed of decay of the small ball probability is faster as either the memory parameter or the space-parameter decreases. \end{abstract}

\maketitle


\section{Introduction}
The analysis of limiting behavior of stochastic processes and random fields has been extensively studied by many authors, including sample path properties \cite{DurastantiMarinucciTodino, BaldiRoynette1992b, LX19,LX23, LWX15, LuMaXi, Yimin, MWX13, T, WX20, Xia07, Xia09}, small probabilities \cite{Carfagnini2022, CarfagniniTodino25, ChenLi2003, KuelbsLi1993a}, and Chung's law of the iterated logarithm \cite{KhoshnevisanShi1998, LX10, Remillard1994, WangXiaoBM, carfagninifoldesherzog2022} for several classes of random processes, including scalar and vector valued random fields and degenerate diffusions. We refer to \cite{NP23, Lif,LiShao2001} and references therein for more complete surveys.

The main purpose of this note is to provide  small ball probabilities and Chung's law of the iterated logarithm for spherical random fields at a critical regime. We recall that a real-valued random field $\{ X_t , t \in \R^{n}\}$ satisfies Chung's law of the iterated logarithm (LIL) at $t_{0} \in \mathbb R^{n}$ with rate function $\phi$ if there exists a constant $C$ such that 
\begin{align}
    & \liminf_{r \rightarrow 0} \phi(r) \max_{t\in B_{r} (t_{0})} \vert X_t - X_{t_{0}} \vert =C  \; \; \; a.s. \label{eqn.chung.LIL}
\end{align}
where $B_{r} (t_{0} )$ is an appropriate ball of radius $r>0$ centered at $t_{0}$.
 The small ball problem (SBP) consists in finding a function $\theta$ and a constant $c$ such that 
\begin{equation}\label{eqn.sd.general}
    \lim_{\varepsilon \rightarrow 0} \theta (\varepsilon) \log \Prob \left( \max_{t\in B_{r} (t_{0})} \vert X_t -X_{t_0} \vert  < \varepsilon \right) =c. 
\end{equation} 
When $X_t$ is a Brownian motion and $t_{0} = 0$, it was proven in a famous paper by K.-L. Chung in 1948 \cite{Chung1948} that \eqref{eqn.chung.LIL} and \eqref{eqn.sd.general} hold with $\phi (r) =  \sqrt{\frac{\log |\log r|}{r}}$, $\theta(\varepsilon)= -\varepsilon^{2}$, and $c=C^{2}= \frac{\pi^{2}}{8}$.

Here, we consider an isotropic centered spherical Gaussian random field $\{ T(x) , \; x \in \mathbb S^{2}\}$ with covariance function given in terms of the power spectrum $\{C_{\ell}, \; \ell =0,1,2,\ldots\}$. Over the past few decades, the analysis of spherical random fields has drawn lot of interest in light of its applications to the analysis of Cosmic Microwave Background radiation data, see \cite{Plank2014, Planck}. Following theoretical and observational evidence on CMB radiation data, the  power spectrum is usually assumed to have the following decay 
\begin{equation}\label{eqn.power.spec.decay}
    C_{\ell} = G(\ell) \ell^{-\alpha}, \quad \ell =1,2,\ldots,
\end{equation}
for some $\alpha >2$, where $G(\ell)$ is a positive bounded function. We refer to Section 2 for more details.

As pointed out in \cite{LX10}, one of the main ingredients in proving Chung's LIL for random fields is a lower bound on the conditional variance, known as strong local nondeterminism. For spherical random fields satisfying condition \eqref{eqn.power.spec.decay} with $2<\alpha <4$, the strong local nondeterminism has been proved in \cite[Theorem 1]{MXL}. More precisely, it is shown that 
\begin{equation}\label{eqn.MLX.variance}
        \text{Var} \left( T(x_{0}) \, | \, T(x_{1} ), \ldots , T(x_{n})   \right) \geqslant c \min_{1\leqslant k \leqslant n} \left(d_{\mathbb{S}^{2}} (x_{0}, x_{k} )\right)^{
        \alpha-2},
    \end{equation}
for all sufficiently close $x_{0},\ldots , x_{n} \in \mathbb{S}^{2}$,  where $c>0$ is a finite constant. The  bound \eqref{eqn.MLX.variance} has then been used to characterize sample paths behavior of spherical random fields, see \cite[Theorem 2]{MXL}. Note that when $2<\alpha<4$, the field's trajectories present a fractal behavior  since the modulus of continuity decays slower than linearly with respect to the angular distance, hence the sample function $T(x)$ is not differentiable  \cite[Theorem 2]{MXL}. In \cite{MXL} it is conjectured that \eqref{eqn.MLX.variance} holds for $\alpha=4$ in the form 
\begin{align}
   & \text{Var} \left( T(x_{0}) \, | \, T(x_{1} ), \ldots , T(x_{n})   \right) \geqslant c \min_{1\leqslant k \leqslant n} d_{\mathbb{S}^{2}} (x_{0}, x_{k} )^{
        2}\vert \log d_{\mathbb{S}^{2}} (x_{0}, x_{k} )\vert.\label{eqn.conjecture.4}
\end{align}

The strong local non-determinism \eqref{eqn.MLX.variance} is a key ingredient in proving small ball estimates and Chung's LIL. It has been used in \cite{LX10} to prove Chung's LIL for stationary Gaussian random fields on $\R^{n}$; and more recently in \cite{CarfagniniTodino25} to prove small fluctuations for time-dependent spherical random fields. In particular, \cite[Theorem 2.11 and Theorem 2.12]{CarfagniniTodino25} prove  \eqref{eqn.chung.LIL} and \eqref{eqn.sd.general} for an isotropic centered spherical Gaussian random field $\{ T(x) , \; x \in \mathbb S^{2}\}$ with power spectrum satisfying \eqref{eqn.power.spec.decay} with $2<\alpha <4$. There, it is shown that  \eqref{eqn.chung.LIL} and \eqref{eqn.sd.general} hold with 
\begin{align*}
    & \phi(r) = \left(\frac{\log \vert \log r \vert}{r^{2}} \right)^{\frac{\alpha-2}{4}}, & \theta (\varepsilon) = - \varepsilon^{\frac{4}{\alpha-2}}.
\end{align*}
We remark that such rates for isotropic spherical random fields   satisfying\eqref{eqn.power.spec.decay}  with $2<\alpha<4$ have been proved originally in \cite{LanXiao18}. Besides Euclidean and spherical case, the strong local nondeterminism, and its applications to the modulus of continuity, has also been studied for istropic Gaussian random fields on compact two-points homogeneous spaces, see \cite{LuMaXi}.

The goal of this note is to prove \eqref{eqn.chung.LIL} and \eqref{eqn.sd.general} for an isotropic centered spherical Gaussian random field in the critical regime, that is, satisfying \eqref{eqn.power.spec.decay} with $\alpha=4$. More precisely, in Theorem \ref{thm.main.new} we show that, if $\alpha=4$ in \eqref{eqn.power.spec.decay} and if \eqref{eqn.conjecture.4} holds, then \eqref{eqn.chung.LIL} and \eqref{eqn.sd.general} are satisfied with 
\begin{align*}
    & \phi(r) = \left(\frac{\log \vert \log r \vert}{r^{2}\vert \log r\vert} \right)^{\frac{1}{2}}, & \theta (\varepsilon) = - \varepsilon^{2} \vert \log \varepsilon \vert^{-1}.
\end{align*}
The extra $\vert \log \varepsilon \vert^{-1}$ and $\vert \log r \vert$ factors which are not present in the $2<\alpha<4$ case come from  the asymptotic of the Lambert function, see Proposition \ref{prop.weak.SD} for more details. Once the correct rates $\phi$ and $\theta$ are identified, the proof of Theorem \ref{thm.main.new} then relies on  Talagrand's bounds for small probabilities of Gaussian fields, and Borel Cantelli Lemmas.

The paper is organized as follows. Section 2 contains background material on spherical random fields and Talagrand bounds on small probabilities for Gaussian random fields. Some preliminary results on the SBP and Chung's LIL are the content of Section 3, and the proof of Theorem \ref{thm.main.new} is in Section 4.

\section{Background and main results}\label{sec2}

\subsection{Talagrand bounds for Gaussian fields} Small ball probabilities of random fields have been extensively studied, see \cite{MWX13,LX,Xia09,X97} and references therein. In this sections we recall results from \cite{T} that we will use throughout the paper. The canonical metric $d_{X}$ associated to a random field $\{ X_{t} , \, t\in \R^{n} \}$ is defined as
 \begin{equation}\label{eqn.canonical.metric}
     d_{X} (t,s) := \sqrt{\E \left[ ( X_{t} - X_{s} )^{2} \right]}.
 \end{equation}
One of the main ingredients in our proofs is Talagrand's bounds on small ball probabilities for Gaussian processes.

\begin{lem} \cite[Lemma 2.1 and Lemma 2.2]{T}\label{lemma.talagrand}
    Let $\{ X_{t} , \, t\in K \}$ be a separable, real-valued, mean zero Gaussian process indexed by a compact set $K$ with canonical metric $d_{X}$. Let $N(K,d_{X} , \varepsilon)$ be the smallest number of $d_{X}$-balls of radius $\varepsilon$ needed to cover the set $K$. Suppose that there is a function $\Psi : (0,\delta ) \rightarrow (0 , \infty)$ such that $N(K,d_{X} , \varepsilon) \leqslant \Psi (\varepsilon)$ for all $\varepsilon \in (0,\delta)$ and there are constants $1< a_{1} \leqslant a_{2}$ such that for all $\varepsilon \in (0,\delta)$
    \begin{align}\label{eqn.condition.talagrand}
        a_{1} \Psi (\varepsilon ) \leqslant \Psi \left(\frac{\varepsilon}{2} \right) \leqslant a_{2} \Psi (\varepsilon).
    \end{align}
    Then, there is a finite constant $k$ depending only on $a_{1}$ and $a_{2}$ such that for all $u\in (0,\delta)$
    \begin{align}
        \mathbb{P} \left( \sup_{s,t \in K} \vert X_{t} - X_{s} \vert \leqslant u \right) \geqslant \exp\left(- k \Psi (u) \right).
    \end{align}
    Moreover, there exists a positive constant $c_{1}$ such that for any $u>0$
    \begin{equation}\label{lemma.talagrand.diameter}
        \mathbb{P} \left( \max_{s,t \in K} \vert X_{s} - X_{t} \vert \geqslant c_{1} \left( u+ \int_{0}^{D} \sqrt{\log N(K,d_{X} , \varepsilon)} d\varepsilon \right) \right) \leqslant \exp \left(- \frac{u^{2}}{D^{2}} \right),        
    \end{equation}
    where  $D$ is the diameter of $K$ with respect to $d_{X}$.
\end{lem}

\subsection{Spherical Gaussian random fields} 

Let $\mathbb{S}^2$ be the unit 2-dimensional sphere and $\Delta_{\mathbb{S}^2}$ be the Laplace-Beltrami operator on $\mathbb{S}^2$ defined as
\begin{equation*}
\dfrac{1}{\sin \theta }\dfrac{\partial }{\partial \theta }\bigg\{\sin \theta 
\dfrac{\partial }{\partial \theta }\bigg\}+\dfrac{1}{\sin ^{2}\theta }\dfrac{%
	\partial ^{2}}{\partial \varphi ^{2}},\mbox{ }0\leq \theta \leq \pi ,\mbox{ }%
0\leq \varphi < 2\pi .
\end{equation*} 
It is well-known that the spectrum of $\Delta_{\mathbb{S}^2}$ is purely discrete and its eigenvalues are of the form $-\lambda_\ell=-\ell(\ell+1)$, where $\ell\in \mathbb N$. The eigenfunctions associated to the eigenvalue $\lambda_\ell$ are the spherical harmonics $\{Y_{\ell m}, \, m=-\ell,\ldots, \ell\}$. Throughout this paper we consider real-valued spherical harmonics, which can be expressed in terms of Legendre associated functions in the following way, see \cite[Remark 3.37]{MP12}
\begin{equation*}
    Y_{\ell m}(\theta,\varphi)= \begin{cases}
\sqrt{\frac{2\ell+1}{2\pi} \frac{(\ell-m)!}{(\ell+m)!}}P_{\ell }^{m} (\cos \theta) \cos(m\varphi) & \mbox{ for } m= 1,\dots,\ell,\\
\sqrt{\frac{2\ell+1}{2\pi}}P_{\ell}^m (\cos \theta)  & \mbox{ for } m=0,\\
\sqrt{\frac{2\ell+1}{2\pi} \frac{(\ell-m)!}{(\ell+m)!}}P_{\ell}^{-m} (\cos \theta) \sin(-m\varphi) & \mbox{ for } m = -\ell,\dots,-1,\\
    \end{cases}
\end{equation*}
where 
\begin{align*}
    & P_\ell:[-1,1]\to \mathbb{R} & P_\ell(t)= \frac{1}{2^{\ell} \ell!} \frac{d^\ell}{dt^\ell}(t^2-1)^\ell 
    \\
    & P_\ell^m(t):[-1,1] \to \mathbb{R}   & P_\ell^m(t):= (1-t^2)^{m/2} \frac{d^m}{dt^m}P_\ell(t)
\end{align*}
for $ \ell=1,2,\dots$ and  $m=-\ell, \ldots, \ell$. The  Legendre polynomial of degree $\ell$, $P_{\ell}$, are normalized so that $P_\ell(1)=1$. Note that the set of spherical harmonics forms a complete orthonormal basis of $L^2(\mathbb S^2, dx)$, where $dx$ denotes the Lebesgue measure. In particular, the functions $Y_{\ell m}$  satisfy
$$\Delta_{\mathbb{S}^2}Y_{\ell m}+\lambda_\ell Y_{\ell m}=0,$$ 
and the addition and duplication formulae read
\begin{equation}
\sum_{m=-\ell }^{\ell }Y_{\ell m}(x){Y}_{\ell m}(y)=\frac{2\ell +1}{4\pi }%
P_{\ell }(\left\langle x,y\right\rangle )\text{ ,}  \label{addition}
\end{equation}%
\begin{equation}
\int_{\mathbb{S}^{2}}\frac{2\ell +1}{4\pi }P_{\ell }(\left\langle
x,z\right\rangle )\frac{2\ell +1}{4\pi }P_{\ell }(\left\langle
z,y\right\rangle )dz=\frac{2\ell +1}{4\pi }P_{\ell }(\left\langle
x,y\right\rangle )\text{ ,}  \label{duplication}
\end{equation}
for all $x,y \in \mathbb{S}^{2}$.

A real-valued spherical random field $T$ is a collection of centered random variables defined on a common probability space $(\Omega, \mathcal{F}, \mathbb{P})$ such that  $T: \Omega \times \mathbb S^{2} \longrightarrow \R$ is an $\mathcal{F}\times \mathcal{F} ( \mathbb S^{2})$-measurable map, where $\mathcal{F} ( \mathbb S^{2} )$ denotes the Borel sigma-algebra of $\mathbb S^{2}$. We are interested in mean-square continuous isotropic centered Gaussian spherical random fields.  We recall that a spherical random field $\{T(x),\, x\in \mathbb{S}^2\}$ is isotropic if $T(x)$ and $T(g . x)$ have the same distribution for any $g\in \text{SO}(3)$. It is well-known that such a field can be represented as, see \cite[Theorem 5.13]{MP12}
\begin{eqnarray}\label{field}
T(x)&=&\sum_{\ell=0}^{\infty}\sqrt{C_\ell}T_\ell(x)=\sum_{\ell=1}^{\infty} \sqrt{C_\ell} \sum_{m=-\ell}^{\ell} a_{\ell,m} Y_{\ell,m}(x), \quad x\in \mathbb S^2,\\
T_\ell(x)&:=&\sum_{m=-\ell}^{\ell} a_{\ell,m} Y_{\ell,m}(x), \quad x\in \mathbb S^2,
\end{eqnarray}
where the random variables $\{a_{\ell,m}, m=-\ell,\dots, \ell\}$ are i.i.d. standard Gaussian. 
The sequence $\{C_\ell\}_{\ell=0,1,\dots}$ is called angular power spectrum of $T(x)$ and it is non-negative and such that $\sum_{\ell} C_\ell \frac{2\ell+1}{4\pi}=\mathbb{E}[T^2]<\infty$ (see \cite{MP12, MP13}). The series  \ref{field} converges in $L^2(\Omega \times \mathbb{S}^2)$, and in   $L^2(\Omega)$ at every fixed $x$. The power spectrum is usually assumed to have the following decay 
\begin{equation}\label{eqn.power.spec.decay2}
    C_{\ell} = G(\ell) \ell^{-\alpha}, \quad \ell =1,2,\ldots,
\end{equation}
where $\alpha>2$ is constant and the function $G(\ell)$ is bounded above and below by positive constants. Condition \eqref{eqn.power.spec.decay2} is aligned with the theoretical and observational evidence on Cosmic Microwave Background radiation data, which has been one of the main motivating areas for the analysis of spherical fields over the last decade. The assumption $\alpha >2$ is needed for the field to have finite variance. Indeed, applying (\ref{addition}) it is easy to see that
$$\mathbb{E}[T_\ell(x)T_\ell(y)]=\frac{2\ell+1}{4\pi}P_\ell(\langle x,y \rangle)$$ and then
the covariance function of $T(x)$ is given by
$$\Gamma(x,y)=\mathbb{E}[T(x)T(y)]= \sum_{\ell=0}^{\infty}  C_\ell \frac{2\ell+1}{4\pi} P_\ell(\langle x,y \rangle) \quad \mbox{ for all } x,y \in \mathbb{S}^2.$$

\begin{assumption}\label{thm.non.determ.1}
    Let $\{ T (x) , \, x\in \mathbb{S}^{2} \}$  be an isotropic Gaussian field with 
    angular power spectrum such that 
\begin{align}
    C_{\ell} = G(\ell) \ell^{-4} >0 , \; \; \ell = 1, 2 , \ldots \label{condition.power.spectrum}
\end{align}
and we assume that 
\begin{align*}
    c^{-1} \leqslant G(\ell ) \leqslant c,
\end{align*}
for some finite constant $c \geqslant 1$. Moreover, we assume that there exist two positive and finite constants $c_{2}$ and $\varepsilon_{0}$ such that for all integers $n\geqslant 1$, and all $x_{0},\ldots , x_{n} \in \mathbb{S}^{2}$ with $\min_{1\leqslant k \leqslant n} d_{\mathbb{S}^{2}} (x_{0} , x_{k}) \leqslant \varepsilon_{0}$ we have that 
    \begin{equation}\label{eqn.bound.variance.1}
        \text{Var} \left( T(x_{0}) \, | \, T(x_{1} ), \ldots , T(x_{n})   \right) \geqslant c_{2} \min_{1\leqslant k \leqslant n} \rho \left(d_{\mathbb{S}^{2}} (x_{0}, x_{k} )\right)^{2},
    \end{equation}
    where 
    \begin{equation}\label{eqn.rho}
    \rho (t) :=  t \sqrt{\vert \ln t \vert } \;\; t>0 , \quad \rho (0):=0.
\end{equation}
\end{assumption}

If \ref{condition.power.spectrum} holds, then by \cite[Lemma 4]{MXL} there exist constants $1\leqslant D <\infty$ and $0 < \varepsilon <1$ such that for all $x,y\in \mathbb{S}^{2}$ with $d_{\mathbb{S}^{2}} (x,y)\leqslant \varepsilon$, one has that 
\begin{align*}
    D^{-1} \rho^{2} \left(d_{\mathbb{S}^{2}} (x,y) \right)\leqslant d_{T}^{2} (x,y) \leqslant D\rho^{2} \left(d_{\mathbb{S}^{2}} (x,y) \right),
\end{align*}
where $\rho$ is given by \eqref{eqn.rho}.

\subsection{Main Results}
Let us now introduce the main object of study and the main result of this paper. 

\begin{notation}\label{not.max}
    For any $x\in \mathbb S^{2}$ and $r>0$ let us set 
    \begin{align*}
       M_{r} (x) := \max_{y\in B_{r} (x)} \vert T(y) - T(x) \vert ,
    \end{align*}
    where $B_{r} (x)$ denotes the geodesic ball in $\mathbb S^{2}$ centered in $x\in \mathbb S^{2}$; and for any $r>0$ let  
\begin{align}\label{eqn.phi}
      \phi (r) :=\left( \frac{\log \vert \log r \vert}{r^{2} \vert \log r \vert} \right)^{\frac{1}{2}} .
\end{align}
\end{notation}

\begin{thm}\label{thm.main.new}
    Let $\{ T (x) , \, x\in \mathbb{S}^{2} \}$  be an isotropic Gaussian field satisfying Assumption \ref{thm.non.determ.1}. Then there exists a finite constant $p>0$ such that 
    \begin{align}
        &  \liminf_{r\rightarrow 0} \phi (r) M_{r} (x) = p, \; \text{a.s.}\label{eqn.Chung.lil.sphere} 
        \\
        & \lim_{\varepsilon \rightarrow 0} - \varepsilon^{2} \vert \log \varepsilon\vert^{-1} \log \mathbb{P} \left( M_{r} (x) < \varepsilon \right) = r^{2} p^{2}.  \label{eqn.small.deviations}
    \end{align}
    for all $x\in \mathbb S^{2}$.
\end{thm}
We remark that the constant $p$ does not depend on $x\in \mathbb S^{2}$. Chung's LIL \eqref{eqn.Chung.lil.sphere} describes sample paths properties of the random field $T(x)$, and due to isotropy such properties do not depend on the fixed point $x\in \mathbb S^{2}$.

\section{Small probabilities}

\subsection{Small balls}
We first prove a weaker version of the small deviation principle \eqref{eqn.small.deviations}. Set 
\begin{align}\label{eqn.psi}
      \psi (r,\varepsilon) =  \frac{r^{2}}{\varepsilon^{2} \vert \log \varepsilon\vert^{-1}}.
\end{align}

\begin{prop}\label{prop.weak.SD}
    Let $\{ T (x) , \, x\in \mathbb{S}^{2} \}$  be an isotropic Gaussian field.

    (i) If \eqref{condition.power.spectrum} holds, then there exists a positive finite constant $A_{1}$ such that for all $\varepsilon >0$ and $r>\varepsilon$ we have that 
    \begin{align}\label{eqn.small.ball.lower.bound}
    \mathbb{P} \left(M_{r} (x) < \varepsilon \right) \geqslant \exp \left( - A_{1} \psi (r, \varepsilon) \right)    ,
    \end{align}

    (ii) If $\{ T (x) , \, x\in \mathbb{S}^{2} \}$ satisfies \eqref{condition.power.spectrum} and \eqref{eqn.bound.variance.1}, then there exists a positive finite constant $A_{2}$ such that for all $\varepsilon >0$ and $r>\varepsilon$ we have that
    \begin{align}\label{eqn.small.ball.upper.bound}
    \mathbb{P} \left( M_{r}(x) < \varepsilon \right) \leqslant \exp \left( -A_{2} \psi (r, \varepsilon) \right)    ,
    \end{align}
    where $\psi (r,\varepsilon)$ is given by \eqref{eqn.psi}.
    \end{prop}

\begin{rmk}
    Note that part (i) in Proposition \ref{prop.weak.SD} does not depend on the strong local nondeterminism \eqref{eqn.bound.variance.1}. Moreover, if the power spectrum has a decay $\ell^{-\alpha}$ for some $2<\alpha <4$, then \eqref{eqn.small.ball.lower.bound} has been proven in \cite{LanXiao18}.
\end{rmk}

    \begin{proof}
        (i) Let $N(B_{r} (x),d_{T} , \varepsilon)$ be the smallest number of $d_{T}$-balls of radius $\varepsilon$ needed to cover  $B_{r} (x)$. The canonical metric $d_{T}$ is equivalent to $\rho (d_{\mathbb{S}^{2}})$, and hence it is enough to consider $N(B_{r} (x),\rho (d_{\mathbb{S}^{2}}) , \varepsilon)$. The $\rho$-ball of radius $\varepsilon$ is given by
    \begin{align*}
        B^{\rho}_{\varepsilon}(x) = \left\{ y \in \mathbb{S}^{2} \, : \, \rho (d_{\mathbb{S}^{2}} (x,y)) < \varepsilon \right\} = \left\{ y \in \mathbb{S}^{2} \, : \, d_{\mathbb{S}^{2}} (x,y) \sqrt{\vert \log d_{\mathbb{S}^{2}} (x,y)\vert }< \varepsilon \right\},
    \end{align*}
    and at the north pole $x=N$ it becomes 
    \begin{align*}
        B^{\rho}_{\varepsilon}(N) = \left\{\left( \varphi , \theta \right)\in [0,2\pi)\times [0,\pi] , \;\; \theta^{2} \vert \log \theta \vert < \varepsilon^{2} \right\}.
    \end{align*}
    Note that the equation $\theta^{2} \vert \log \theta \vert = \varepsilon^{2}$ has three distinct real solutions for $\varepsilon$ small, $0<\theta_{1}(\varepsilon)<\theta_{2}(\varepsilon)<1< \theta_{3}(\varepsilon)$ such that
    \begin{align*}
        & \theta_{1}(\varepsilon) \rightarrow 0, & \theta_{2}(\varepsilon) \rightarrow 1, & & \theta_{3}(\varepsilon) \rightarrow 1,
    \end{align*}
    as $\varepsilon \rightarrow 0$. Due to isotropy, we can assume $x=N$, and since we are interested in a small ball asymptotics around $x=N$, we have that $\theta=d_{\mathbb{S}^{2}}(y,N) <<1$.  Thus, we can write
    \begin{align*}
        & \text{Vol}\left( B^{\rho}_{\varepsilon}(N)\right) =  2\pi \int_{0}^{\theta_{1}(\varepsilon)} \sin \theta d\theta = 2\pi \left(1-\cos \theta_{1}(\varepsilon) \right) = \pi  \theta_{1} (\varepsilon)^{2} + o \left(  \theta_{1} (\varepsilon)^{2} \right).
    \end{align*}

    By Lemma \ref{l.approx} we have that 
    \begin{align*}
        & \theta_{1} (\varepsilon)^{2}  = 4\varepsilon^{2} \vert \log \varepsilon\vert^{-1} +o\left( \varepsilon^{2} \vert \log \varepsilon\vert^{-1} \right) ,
    \end{align*}
    that is,
    \begin{align*}
        & \text{Vol}\left( B^{\rho}_{\varepsilon}(N)\right) =  4\pi\varepsilon^{2} \vert \log \varepsilon\vert^{-1} +o\left( \varepsilon^{2} \vert \log \varepsilon\vert^{-1} \right),
    \end{align*}
    and hence 
    \begin{align*}
        N(B_{r} (N),\rho (d_{\mathbb{S}^{2}}) , \varepsilon)\leqslant \frac{\text{Vol} \left( B_{r} (N) \right)}{\text{Vol} (B^{\rho}_{\varepsilon}(N))} = \frac{1-\cos r}{4\pi\varepsilon^{2} \vert \log \varepsilon\vert^{-1} +o\left( \varepsilon^{2} \vert \log \varepsilon\vert^{-1} \right)} =\frac{r^2}{\varepsilon^{2} \vert \log \varepsilon\vert^{-1}} F(r,\varepsilon),
    \end{align*}
    where $F(r,\varepsilon)$ is uniformly bounded in $\varepsilon$ and $r$. Thus we can take 
    \begin{align*}
        \psi(r,\varepsilon):= \frac{r^2}{\varepsilon^{2} \vert \log \varepsilon\vert^{-1}} ,
    \end{align*}
    which satisfies the assumption of Lemma \ref{lemma.talagrand}.

    (ii)  The proof relies on a conditioning argument similar to the proof of \cite[Proposition 4.2]{LX}. Let $F: A_{r} \rightarrow B_{r} (x)$  be a local diffeomorphism, where $A_{r} \subset \mathbb{R}^{2}$ is a rectangle and $F(0,0) = x$. Let us consider a rectangle $I= [0, r] \times [0, r]$ for $r$ small enough so that $I\subset A_{r}$, and let us  divide $I$ into sub-rectangles of side-length
    \[
        \frac{r}{\bigg\lfloor \frac{r}{\varepsilon \vert \log \varepsilon \vert^{-\frac{1}{2}}} \bigg \rfloor +1} ,
    \]
    where $\lfloor \cdot \rfloor$ denotes the floor function. The number of such subrectangles is
    \[
    N =  \left( \bigg\lfloor \frac{r}{\varepsilon  \vert \log \varepsilon \vert^{-\frac{1}{2}}} \bigg \rfloor +1 \right)^{2} \geqslant \frac{r^{2}}{\varepsilon^{2} \vert \log \varepsilon \vert^{-1}} 
    \]
    Let $(\theta_{i} , \varphi_{i} )$ for $i=1,\ldots , N$ denote the upper right-most vertex of the subrectangle, and set $x_{i}:= F(\theta_{i}, \varphi_{i})$. Then by construction 
    \begin{align*}
         d_{\mathbb{R}^{2}} \left( (\theta_{i}, \varphi_{i}) , (\theta_{j}, \varphi_{j})  \right) \geqslant c \, \frac{r}{\bigg\lfloor \frac{r}{\varepsilon  \vert \log \varepsilon \vert^{-\frac{1}{2}}} \bigg \rfloor +1} \quad \text{ for all } i,j=1,\ldots , N
    \end{align*}
    for some constant $c>0$. We can choose $F$ in such a way that $d_{\mathbb{S}^{2}} (x_{i} , x_{j} ) \geqslant d_{\mathbb{R}^{2}} \left( (\theta_{i}, \varphi_{i}) , (\theta_{j}, \varphi_{j})  \right)$. Set 
    \begin{align*}
        A_{j} := \left\{ \max_{i=1, \ldots , j} \vert T(x_{i}) - T(x) \vert < \varepsilon \right\},
    \end{align*}
    and write 
    \begin{align*}
        \mathbb{P} (A_{j} ) = \mathbb{E} \left[ 1_{A_{j-1}} \, \mathbb{P} \left( \vert T(x_{j}) - T(x) \vert < \varepsilon \, \left. \right| T(x_{i}) \, : \, i=0, \ldots , j-1  \right) \right].
    \end{align*}
    By \eqref{eqn.bound.variance.1} we have that 
    \begin{align*}
         &   \text{Var} \left( T(x_{j}) \, | \, T(x_{i}) \, : \, i=0, \ldots , j-1   \right) \geqslant c_{2} \min_{1\leqslant i \leqslant j-1} \rho \left(d_{\mathbb{S}^{2}} (x_{j}, x_{i}) \right)^{2} 
         \\
         & \geqslant c_{2} \, \rho \left(  \frac{c\;r}{\bigg\lfloor \frac{r}{\varepsilon  \vert \log \varepsilon \vert^{-\frac{1}{2}}} \bigg \rfloor +1} \right)^{2} \geqslant c_{2} \,  \rho \left(  \frac{c\;r}{\frac{2r}{\varepsilon  \vert \log \varepsilon \vert^{-\frac{1}{2}}}} \right)^{2}  = c_{2} \,  \rho \left( \frac{c}{2} \varepsilon  \vert \log \varepsilon \vert^{-\frac{1}{2}} \right)^{2}  
         \\
         & = \frac{c_{2}}{4}c^{2} \varepsilon^{2}  \vert \log \varepsilon \vert^{-1} \left| \log \left(\frac{c}{2}\varepsilon  \vert \log \varepsilon \vert^{-\frac{1}{2}} \right) \right| \geqslant D \varepsilon^{2},
    \end{align*}
    where we used that $\rho(t)$ is an increasing function for small values of $t>0$, and that 
    \begin{align*}
        \lim_{\varepsilon\rightarrow 0}  \vert \log \varepsilon \vert^{-1} \left| \log \left(\frac{c}{2}\varepsilon  \vert \log \varepsilon \vert^{-\frac{1}{2}} \right) \right|=1.
    \end{align*}
    The conditional distribution of $T(x_{j})$ given all the $T(x_{i})$ is Gaussian with conditional variance $ \text{Var} \left( T(x_{j}) \, | \, T(x_{i}) \, : \, i=0, \ldots , j-1   \right)$ since the field $T(x)$ is Gaussian. Then by the above lower bound on the conditional variance and Anderson's inequality \cite{Anderson}  it follows that 
    \begin{align*}
        & \mathbb{P} \left( \vert T(x_{j}) - T(x) \vert < \varepsilon \, \left. \right| T(x_{i}) \, : \, i=0, \ldots , j-1  \right) \leqslant \mathbb{P} \left( \vert Z \vert \leqslant \frac{1}{\sqrt{D}} \right)\leqslant \exp(-C),
    \end{align*}
    where $Z$ is a standard Gaussian random variable and the last inequality holds for some constant $C>0$. Thus,
    \begin{align*}
        \mathbb{P} (A_{N}) \leqslant \exp(-C)  \mathbb{P} (A_{N-1}) \leqslant \cdots  \leqslant \exp( -NC),
    \end{align*}
    and hence 
    \begin{align*}
         & \mathbb{P} \left( M_{r} (x) < \varepsilon \right) \leqslant  \mathbb{P} \left( \max_{i=1, \ldots , N} \vert T(x_{i}) - T(x)) \vert < \varepsilon \right) = \mathbb{P} (A_{N}) 
         \\
         & \leqslant \exp( -NC) \leqslant \exp \left( - C \frac{r^{2}}{\varepsilon^{2}\vert \log \varepsilon \vert^{-1}}  \right).
    \end{align*}
    \end{proof}

\subsection{Chung's LIL}

The goal of this section is to prove a weaker version of \eqref{eqn.Chung.lil.sphere}. Let us recall the following notation 

\begin{align*}
   & p:= \liminf_{r\rightarrow 0} \phi (r) M_{r} (x), & \phi(r) := \sqrt{\frac{\log \vert \log r \vert}{r^{2}\vert \log r \vert}}, & & M_{r} (x) := \max_{y \in B_{r} (x)} \vert T(y) - T(x) \vert  .
\end{align*}

\begin{prop}\label{prop.weak.Chung}
     Let $\{ T (x) , \, x\in \mathbb{S}^{2} \}$  be an isotropic Gaussian field.

    (i) If \eqref{condition.power.spectrum} holds, then 
        \begin{equation}\label{eqn.chung.lower.bound}
            p \leqslant 2\sqrt{A_{1}}  \; \; a.s. 
        \end{equation}
       where $A_{1}$ is given by \eqref{eqn.small.ball.lower.bound}.

    (ii) If $\{ T (x) , \, x\in \mathbb{S}^{2} \}$ satisfies  \eqref{condition.power.spectrum} and \eqref{eqn.bound.variance.1}, then 
    \begin{equation}\label{eqn.chung.lower.bound}
         p \geqslant \sqrt{A_{2}}  \; \; a.s. 
        \end{equation}
        where $A_{2}$ is given by \eqref{eqn.small.ball.upper.bound}.
\end{prop}

\begin{proof}

    (i) 
    Let us set 
    \begin{align*}
        & r_{n} := e^{-n} \lfloor e^{n^{2}} \rfloor^{-1}, & d_{n} :=\lfloor e^{n^{2}} \rfloor ,
    \end{align*}
    and note that 
    \begin{align*}
        &r_{n} d_{n} = e^{-n} &  r_{n} d_{n+1} > e^{n}.
    \end{align*}
    Let us consider the random fields
    \begin{align*}
        & \Xi_{n} (x):= \sum_{\ell =d_{n}+1}^{d_{n+1}} \sqrt{C_{\ell}}  T_{\ell} (x)  , & \Theta_{n} (x) := T(x) -  \Xi_{n} (x) =  \sum_{\ell = 1}^{d_{n}} \sqrt{C_{\ell}}  T_{\ell} (x)  + \sum_{\ell = d_{n+1} +1}^{\infty} \sqrt{C_{\ell}} T_{\ell} (x)
    \end{align*}
    For a fixed $x\in \mathbb S^{2}$, the Gaussian random fields $\{\Xi_{n} (x)  \}_{n=1}^{\infty}$ are independent from each other and moreover, for each $n$ and $x$, $\Xi_{n} (x)$ is independent of $\Theta_{n} (x)$. It is enough to show that 
    \begin{align}
        & \sum_{n=1}^{\infty} \mathbb{P} \left( \phi (r_{n}) \max_{y\in B_{r_{n}} (x)} \vert  \Xi_{n} (y)  - \Xi_{n}(x) \vert \leqslant \gamma \right) = \infty , \quad \text{ for any } \gamma >2\sqrt{A_{1}},  \label{eqn.almost.there.1}
        \\
        & \sum_{n=1}^{\infty} \mathbb{P} \left( \phi (r_{n}) \max_{y\in B_{r_{n}} (x)} \vert \Theta_{n} (y) - \Theta_{n} (x) \vert > \varepsilon \right) < \infty , \quad \text{ for any } \varepsilon >0. \label{eqn.almost.there.2}
    \end{align}
    Indeed, by Borel-Cantelli Lemmas and \eqref{eqn.almost.there.1} \eqref{eqn.almost.there.2} it follows that 
    \begin{align*}
       &  \liminf_{r\rightarrow 0} \phi (r) M_{r} (x) \leqslant \liminf_{n\rightarrow \infty } \phi (r_{n}) M_{r_{n}} (x) 
       \\
       & \leqslant  \liminf_{n\rightarrow \infty } \phi (r_{n})  \max_{y\in B_{r_{n}} (x)} \vert \Xi_{n} (y) -\Xi_{n} (x) \vert  + \liminf_{n\rightarrow \infty } \phi (r_{n})  \max_{y\in B_{r_{n}} (x)} \vert \Theta_{n} (y) - \Theta_{n} (x) \vert \leqslant \gamma + \varepsilon,
    \end{align*}
    for any $\gamma > 2\sqrt{A_{1}}$ and $\varepsilon >0$. Let us first prove \eqref{eqn.almost.there.1}. By Anderson inequality \cite{Anderson} and Proposition \ref{prop.weak.SD} part (i), it follows that 
    \begin{align*}
         & \mathbb{P} \left( \phi (r_{n}) \max_{y\in B_{r_{n}} (x)} \vert \Xi_{n} (y) - \Xi_{n}(x) \vert \leqslant \gamma \right)   \geqslant   \mathbb{P} \left( \phi (r_{n}) \max_{y\in B_{r_{n}} (x)} \vert \Xi_{n} (y) -  \Xi_{n}(x) + \Theta_{n} (y) -\Theta_{n} (x) \vert \leqslant \gamma \right) 
         \\
         & =  \mathbb{P} \left( \phi (r_{n}) M_{r_{n}} (x) \leqslant \gamma \right) \geqslant \exp \left( - A_{1} \psi \left( r_{n} , \frac{\gamma}{\phi(r_{n})} \right)\right) 
         \\
         &=\exp\left( -\frac{A_{1}}{\gamma^{2}} \log a_{n} \left|1+\frac{1}{a_{n}} \log \gamma + \frac{1}{2a_{n}} \log \frac{ a_{n}}{\log a_{n}} \right| \right) \geqslant \exp\left( -\frac{4A_{1}}{\gamma^{2}} \log a_{n}  \right) = a_{n}^{-\frac{4A_{1}}{\gamma^{2}}},
    \end{align*}
    where $a_{n} := \vert \log r_{n} \vert$,    and \eqref{eqn.almost.there.1} is proven since $\gamma > 2\sqrt{A_{1}}$.

    We  now prove \eqref{eqn.almost.there.2} by means of \eqref{lemma.talagrand.diameter} with $K= B_{r} (x)$ and $X_{t} =\Theta_{n} (x)$. 
    First, let us bound the diameter $D$ of $B_{r}(x)$ with respect to the distance $d_{\Theta_{n}}$. By the addition formula for spherical harmonics we have that
    \begin{align*}
        &\max_{y\in B_{r} (x)} \mathbb{E} \left[ ( \Theta_{n} (y) - \Theta_{n} (x) )^{2} \right] 
        \\
        & = \max_{y\in B_{r} (x)} \left(  \sum_{\ell =1}^{d_{n}} C_{\ell} \sum_{m=-\ell}^{\ell} \left( Y_{\ell m} (x) - Y_{\ell m} (y)\right)^{2} + \sum_{\ell= d_{n+1}+1}^{\infty} C_{\ell} \sum_{m=-\ell}^{\ell} \left( Y_{\ell m} (x) - Y_{\ell m} (y)\right)^{2} \right)
        \\
        & =  \max_{y\in B_{r} (x)} \left(  \sum_{\ell = 1}^{d_{n}} C_{\ell} \frac{2\ell +1}{2\pi} (1- P_{\ell} (\cos \theta) ) +  \sum_{\ell = d_{n+1}+1}^{\infty} C_{\ell} \frac{2\ell +1}{2\pi} (1- P_{\ell} (\cos \theta) )  \right)  \leqslant S_{1} + S_{2}, 
    \end{align*}
    where $\theta:= d_{\mathbb S^{2}} (x,y)$ and 
    \begin{align*}
        &S_{1} : = \max_{y\in B_{r} (x)}  \sum_{\ell = 1}^{d_{n}} C_{\ell} \frac{2\ell +1}{2\pi} (1- P_{\ell} (\cos \theta) ) , & S_{2} := \max_{y\in B_{r} (x)} \sum_{\ell = d_{n+1}+1}^{\infty} C_{\ell} \frac{2\ell +1}{2\pi} (1- P_{\ell} (\cos \theta) ).
    \end{align*}
    Note that $\vert P_{\ell} (\cos \theta) \vert \leqslant 1$ and hence, by  \eqref{condition.power.spectrum},
    \begin{align*}
        & S_{2} =  \max_{y\in B_{r} (x)} \sum_{\ell = d_{n+1}+1}^{\infty} C_{\ell} \frac{2\ell +1}{2\pi} (1- P_{\ell} (\cos \theta) ) \leqslant  \sum_{\ell = d_{n+1}+1}^{\infty} C_{\ell} \frac{2\ell +1}{\pi}
        \\
        & \leqslant c \sum_{\ell = d_{n+1}}^{\infty} \ell^{-3} \leqslant c \int_{d_{n+1}}^{\infty} x^{-3}dx = c\,  \left( \frac{1}{d_{n+1}} \right)^{ 2} \leqslant c \, \left(r_{n} e^{-n} \right)^{2}.
    \end{align*}
    By Lemma \ref{lemma.taylor.expansion.Pl} it follows that 
    \begin{align*}
        & S_{1}  =  \max_{y\in B_{r} (x)}  \sum_{\ell = 1}^{d_{n}} C_{\ell} \frac{2\ell +1}{2\pi} (1- P_{\ell} (\cos \theta) ) \leqslant   \sum_{\ell = 1}^{d_{n}} C_{\ell} \frac{2\ell +1}{2\pi}  \max_{y\in B_{r} (x)} (1- P_{\ell} (\cos \theta) ) 
        \\
        & \leqslant c  \sum_{\ell = 1}^{d_{n}} \ell^{-3} ( \ell^{2} r_{n}^{2} + \ell^{4} r_{n}^{4} ) = c r_{n}^{2}   \sum_{\ell = 1}^{d_{n}} \ell^{-1} +   c r_{n}^{4}   \sum_{\ell = 1}^{d_{n}} \ell 
        \\
        & \leqslant c ( r_{n}^{2} \log d_{n} +  r_{n}^{4} d_{n}^{2}  ) =   c r_{n}^{2} \left( \log d_{n} +(r_{n} d_{n} )^{2} \right)
        \\
        & = c r_{n}^{2} \left( \log d_{n} + e^{-2n} \right) \leqslant  c r_{n}^{2} \log d_{n} \leqslant c\, n \, r_{n}^{2} . 
    \end{align*}
    Combining everything together it follows that 
    \begin{align*}
        D^{2} \leqslant S_{1} + S_{2} \leqslant    c\, n \, r_{n}^{2} .
    \end{align*}
     If $N_{\varepsilon}$ denotes the number of $d$-balls of radius $\varepsilon$ needed to cover $B_{r_{n}}(x)$, then proceeding as in the proof of Proposition \ref{prop.weak.SD} we have that 
    \[
    N_{\varepsilon} \leqslant \frac{r_{n}^{2}}{\varepsilon^{2}\vert \log \varepsilon \vert^{-1}}.
    \]
    Let us set $t:= \frac{\varepsilon^{2}}{r_{n}^{2}} \vert \log \varepsilon\vert^{-1}$, that is, $\varepsilon= \exp\left(\frac{1}{2} W_{-1}(-2tr_{n}^{2})\right)$, where $W_{-1}$ is the Lambert function, see the Appendix for more details. Then
    \begin{align*}
        & \int_{0}^{D} \sqrt{\log N_{\varepsilon}} d\varepsilon \leqslant \int_{0}^{ c\, \sqrt{n} \, r_{n} } \sqrt{\log \frac{r_{n}^{2}}{\varepsilon^{2}\vert \log \varepsilon \vert^{-1}} } d\varepsilon
        \\
        & = -\int_{0}^{c \, n \vert\log (c\, r_{n} \sqrt{n}) \vert}\sqrt{\vert \log t \vert} \exp\left(\frac{1}{2} W_{-1}(-2tr_{n}^{2})\right) \frac{1}{2} W_{-1}^{\prime}(-2tr_{n}^{2}) 2r_{n}^{2}dt
        \\
        & = \int_{0}^{c \, n \vert\log (c\, r_{n} \sqrt{n}) \vert}\sqrt{\vert \log t \vert}\exp\left(\frac{1}{2} W_{-1}(-2tr_{n}^{2})\right)\frac{1}{2t} \frac{W_{-1}(-2tr_{n}^{2})}{1+ W_{-1}(-2tr_{n}^{2})} dt,
    \end{align*}
    where we used \eqref{eqn.prop.lambert}. Then, by \eqref{eqn.prop.lambert} again it follows that 
    \begin{align}\label{eqn.const2}
        & \int_{0}^{D} \sqrt{\log N_{\varepsilon}} d\varepsilon \leqslant r_{n} \int_{0}^{c \, n \vert\log (c\, r_{n} \sqrt{n}) \vert}   \sqrt{\frac{\vert \log t\vert}{2t}}dt \leqslant c_{2}\,\sqrt{\frac{r_{n}^{2}}{n^3} \log n},
    \end{align}
    where the last inequality follows by Lemma \ref{lemma.integral.bound}. Set
    \begin{align*}
        u:= \frac{\varepsilon}{c_{1}\phi(r_{n})}  - c_{2}\,\sqrt{\frac{r_{n}^{2}}{n^3} \log n},
    \end{align*}
    where $c_{1}$ and $c_{2}$ are given by \eqref{lemma.talagrand.diameter} and \eqref{eqn.const2} respectively, and note that, for all $n$ sufficiently large
    \begin{align}\label{eqn.bound.u}
        & u> \frac{A}{\phi(r_{n})},
    \end{align}
    for some positive constant $A$. Then by   \eqref{lemma.talagrand.diameter}
    \begin{align*}
        & \mathbb{P} \left( \phi (r_{n}) \max_{y\in B_{r_{n}} (x)} \vert \Theta_{n} (y) - \Theta_{n}(x) \vert > \varepsilon \right)  \leqslant \mathbb{P} \left(  \max_{z, y\in B_{r_{n}} (x)} \vert  \Theta_{n} (y) -  \Theta_{n} (z) \vert > \frac{\varepsilon}{\phi (r_{n})} \right)
        \\
        & = \mathbb{P} \left(  \max_{z, y\in B_{r_{n}} (x)} \vert \Theta_{n} (y) - \Theta_{n} (z) \vert > c_{1} \left( u+  c_{2}\,\sqrt{\frac{r_{n}^{2}}{n^3} \log n} \right) \right)
        \\
        & \leqslant \mathbb{P} \left(  \max_{z, y\in B_{r_{n}} (x)} \vert \Theta_{n} (y) - \Theta_{n} (z) \vert > c_{1} \left( u+  \int_{0}^{D} \sqrt{\log N_{\varepsilon}} d\varepsilon \right) \right)
        \\
        &  \leqslant \exp \left( - \frac{u^{2}}{D^{2}}\right) \leqslant \exp \left( - \frac{A^{2}}{\phi (r_{n})^{2}D^{2}}\right)\leqslant \exp \left( - \frac{c}{\phi (r_{n})^{2}n r_{n}^{2}}\right) = \exp\left( - \frac{c (1+n)}{\log (n+n^2)} \right),
    \end{align*}
    which completes the proof since 
    \begin{align*}
        \sum_{n=1}^{\infty} \exp\left( - \frac{c (1+n)}{\log (n+n^2)} \right) < \infty.
    \end{align*}

    (ii) For $n \geqslant 1$ and $R>1$ let us set 
    \begin{align*}
        & r_{n} := R^{-n} , & 0 < \gamma <  \frac{ \sqrt{ A_{2}} }{2R \sqrt{\log R}}  , & & A_{n} := \left\{ \phi (r_{n-1}) M_{r_{n}}(x)  \leqslant \gamma \right\}.
    \end{align*}
    Then by \eqref{eqn.small.ball.upper.bound} we have that 
    \begin{align*}
        & \mathbb{P} (A_{n} ) \leqslant \exp\left(-A_{2} \psi \left( r_{n} , \frac{\gamma}{\phi(r_{n-1})} \right) \right) 
        \\
        & = \exp\left( - \frac{A_{2}}{R^{2} \gamma^{2} \log R} \log \left((n-1)\log R\right) \left| 1- \frac{1}{n-1} \log \gamma+ \frac{1}{2(n-1)} \log \frac{\log \left( (n-1)\log R \right)}{(n-1)\log R} \right|   \right)
        \\
        & \leqslant \exp\left( - \frac{A_{2}}{4R^{2} \gamma^{2} \log R} \log \left((n-1)\log R\right)   \right)
        \\
        &= ((n-1) \log R)^{-\frac{A_{2}}{4\gamma^{2} R^{2} \log R}},
    \end{align*}
    for all $n$ large enough. Thus,  $\sum_{n=1}^{\infty} \mathbb{P} (A_{n} ) < \infty$ since $A_{2} >4 \gamma^{2} R^{2} \log R$. Note that $\phi (s) > \phi (t)$ for all $s<t$ small, and then 
    \begin{align*}
        \inf_{r_{n} \leqslant r \leqslant r_{n-1}} \phi (r) M_{r} (x) \geqslant \phi(r_{n-1}) M_{r_{n}}(x) .
    \end{align*}
Note that, for any $ 0 < \gamma <  \frac{ \sqrt{A_{2}}  }{2R \sqrt{\log R}} $
\begin{align*}
    & \mathbb{P} \left( \liminf_{r\rightarrow 0} \phi (r) M_{r} (x)< \gamma \right) \leqslant \mathbb{P} \left( \bigcup_{k\geqslant 1} \bigcap_{n\geqslant k} \left\{ \inf_{r_{n} \leqslant r \leqslant r_{n-1}} \phi (r) M_{r} (x)  < \gamma \right\}         \right)
    \\
    &  \leqslant \mathbb{P} \left( \bigcup_{k\geqslant 1} \bigcap_{n\geqslant k} \left\{ \phi(r_{n-1})M_{r_{n}}(x)  < \gamma \right\}         \right) = \mathbb{P} \left( \liminf_{n\rightarrow \infty} A_{n} \right) =0,
\end{align*}
    where the latter is zero by the first Borel-Cantelli Lemma. Thus, we have that 
    \begin{align*}
        \liminf_{r\rightarrow 0} \phi (r) M_{r} (x)  \geqslant \gamma \; \; a.s. \quad \text{ for every }  0 < \gamma < \frac{ \sqrt{A_{2}}  }{2R \sqrt{\log R}},
    \end{align*}
    and the result follows first by letting $R$ go $R^{\ast}$, where $R^{\ast}$ solves $2R^{\ast}\sqrt{\log R^{\ast}}=1$, and then  taking the supremum over $\gamma$ on both sides.

\end{proof}

\section{Proof of Theorem \ref{thm.main.new}}
    Let us set 
    \begin{align*}
        & c_{+} := \limsup_{\varepsilon \rightarrow 0} - \frac{1}{\psi(r,\varepsilon)} \log \mathbb{P} \left( M_{r}(x) < \varepsilon \right) ,
        & c_{-} := \liminf_{\varepsilon \rightarrow 0} - \frac{1}{\psi(r,\varepsilon)} \log \mathbb{P} \left( M_{r}(x) < \varepsilon \right) ,
    \end{align*}
    where $\psi$ is given by \eqref{eqn.psi}. By Proposition \ref{prop.weak.SD} we know that 
    \[
    0<A_{2} \leqslant c_{-} \leqslant c_{+} \leqslant A_{1} < \infty .
    \]
    The strategy of the proof is to show that 
    \begin{align*}
        c_{+} \leqslant   p^{2} \leqslant c_{-} ,
    \end{align*}
    where $p:= \liminf_{r\rightarrow 0} \phi(r) M_{r}(x)$.
    Let us first prove that $c_{+} \leqslant   p^{2}$. Let us fix $k\in (0,c_{+})$. Then there exists an $\varepsilon (k)$ such that 
    \begin{equation}\label{eqn.sd.upper.bound}
            \mathbb{P} \left(
            M_{r}(x)< \varepsilon \right)  \leqslant \exp \left( -k \,\psi (r,\varepsilon) \right),
    \end{equation}
    for any $\varepsilon \leqslant \varepsilon (k)$. Let $r_{n} := R^{-n}$ for some $R>1$ and let us choose $\gamma$ such that $0<4  \gamma^{2} R^{2} \log R < k$. Then, proceeding like in the proof of Proposition \ref{prop.weak.SD} it follows that 
    \begin{align*}
        &  \mathbb{P} \left(M_{r_{n}} (x) < \frac{\gamma}{\phi(r_{n-1})} \right) \leqslant \left( \frac{1}{(n-1) \log R} \right)^{\frac{k}{4  \gamma^{2} R^{2} \log R}},
    \end{align*}
    and hence 
    \begin{align*}
        & \sum_{n=1}^{\infty} \mathbb{P} \left( M_{r_{n}}(x)< \frac{\gamma}{\phi(r_{n-1})} \right)  < \infty,
    \end{align*}
    since $4  \gamma^{2} R^{2} \log R < k$. Hence by Borel-Cantelli Lemma we have that, almost surely for all large $n$ 
    \[
     M_{r_{n}}(x)\geqslant \frac{\gamma}{\phi(r_{n-1})}.
    \]
    The function $\phi(t)$ is decreasing for $t<<1$, and hence for $r_{n} \leqslant r \leqslant r_{n-1}$ we have that 
    \begin{align*}
        & M_{r}(x) \geqslant M_{r_{n}}(x)  \geqslant \frac{\gamma}{\phi(r_{n-1})} \geqslant   \frac{\gamma}{\phi(r)},
    \end{align*}
    which yields  
    \begin{align*}
        p := \liminf_{r\rightarrow 0} \phi(r) M_{r}(x) \geqslant  \gamma \; a.s.
    \end{align*}
    for any $\gamma < \sqrt{ \frac{k}{4R^{2}\log R} }<\sqrt{ \frac{c_{+}}{4R^{2}\log R} }$. The estimate  $c_{+} \leqslant   p^{2}$ then follows by first letting $R$ go to $R^{\ast}$, where $R^{\ast}$ solves $4R^{2}\log R=1$, then letting  $k$ to $c_{+}$, and lastly taking the supremum over $\gamma$.

    Let us now prove that $p^{2} \leqslant c_{-}$. Let $k>c_{-}$ be fixed. Then there exists an $\varepsilon(k)$ such that 
    \begin{align*}
            \mathbb{P} \left( M_{r}(x) < \varepsilon \right)  \geqslant \exp \left( -k \psi (r,\varepsilon) \right),
    \end{align*}
    for any $\varepsilon \leqslant \varepsilon (k)$. Let $\Xi_{n} (x) $, and $\Theta_{n} (x)$ be defined as in the proof of Proposition \ref{prop.weak.Chung}, and  $\gamma$ be a real number such that $k< \gamma^{2}$. Set $s_{n} := \exp\left( -\vert \log r_{n} \vert^{\frac{1}{4}} \right)$, and note that 
    \begin{align*}
        \vert \log r_{n}\vert = \vert \log s_{n}\vert^{4}.
    \end{align*}
    Then, proceeding as in the proof of Proposition \ref{prop.weak.Chung} we have that 
    \begin{align*}
        & \mathbb{P} \left( \phi (s_{n} ) \max_{y\in B_{s_{n}} (x)} \vert \Xi_{n} (y) - \Xi_{n}(x) \vert \leqslant\gamma \right)  \geqslant \mathbb{P} \left(  M_{s_{n}}(x) \leqslant \frac{\gamma}{\phi (s_{n} )} \right) 
        \\
        &\geqslant \exp \left( -\frac{4k}{\gamma^{2}} \log \vert\log s_{n} \vert \right) = \exp \left( -\frac{k}{\gamma^{2}} \log \vert\log r_{n} \vert \right) \geqslant \left( \frac{1}{ \vert \log r_{n} \vert } \right)^{\frac{k}{\gamma^{2}}} ,
    \end{align*}
    The events 
    \[
    \left\{  \phi (s_{n} ) \max_{y\in B_{s_{n}} (x)} \vert \Xi_{n} (y) -\Xi_{n}(x) \vert \leqslant \gamma  \right\}
    \]
    are independent, and hence by Borel-Cantelli Lemma for independent events it follows that 
    \[
    \liminf_{n\rightarrow \infty}   \phi (s_{n} ) \max_{y\in B_{s_{n}} (x)} \vert \Xi_{n} (y) -\Xi_{n} (x)\vert \leqslant \gamma  \, a.s. \, 
    \]
    for any  $\gamma^{2}> k >  c_{-}$, and hence, 
    \[
    \liminf_{n\rightarrow \infty}   \phi (s_{n} ) \max_{y\in B_{s_{n}} (x)} \vert\Xi_{n} (y) -\Xi_{n}(x) \vert \leqslant\sqrt{ c_{-}} \, a.s. \, 
    \]
    Note that as $n\rightarrow \infty$, the sequence $\frac{s_{n}}{r_{n}}$ stays bounded, ans hence \eqref{eqn.almost.there.2}  holds with $r_{n}$ replaced by $s_{n}$. Thus,  by \eqref{eqn.almost.there.2} we have that 
    \begin{align*}
        & p := \liminf_{r \rightarrow 0}   \phi (r )M_{r} (x) \leqslant  \liminf_{n\rightarrow \infty}   \phi (s_{n} )M_{s_{n}}(x)
        \\
        & \leqslant  \liminf_{n\rightarrow \infty}   \phi (s_{n} ) \max_{y\in B_{s_{n}} (x)} \vert\Xi_{n} (y)-\Xi_{n}(x) \vert + \liminf_{n\rightarrow \infty}   \phi (s_{n} ) \max_{y\in B_{s_{n}} (x)} \vert \Theta_{n} (y) -\Theta_{n}(x) \vert
        \\
        & \leqslant \sqrt{ c_{-}}  +\varepsilon
    \end{align*}
    for any $\varepsilon>0$, and hence the result follows by letting $\varepsilon$ go to zero.

\section{Appendix}

\subsection{Lambert function}\label{sub.sec.lambert}
We recall here the definition and basic properties of the Lambert function $W(x)$. For a given $x \geqslant -e^{-1}$, the Lambert function $W_{k}$, $k=0,1$, is defined as the solution to the equation 
\begin{equation}\label{eqn.lambert}
    y e^{y} =x.
\end{equation}
 The solution to equation \eqref{eqn.lambert} is denoted by 
 \begin{align*}
     & W_{0} (x), & x\geqslant 0, 
     \\
     &W_{-1} (x), & -e^{-1} \leqslant x <0.
 \end{align*}
The following properties are well known
\begin{align}
    & W_{-1}^{\prime} (x) = \frac{W_{-1}(x)}{x \left(1+ W_{-1}(x) \right)}, \text{ for any } -e^{-1} < x <0, & \lim_{x\rightarrow 0^{+}} \frac{W_{-1} (-x)}{\log x} =1. \label{eqn.prop.lambert}
    \\
    & W_{0} (x) = x+ \text{o} (x) \, \text{ as } x\rightarrow 0.  \nonumber
\end{align}

\begin{lem}\label{l.approx}
    Let $\theta_{i} (\varepsilon)$ for $i=1,2,3$ be the three distinct real solutions to 
    \begin{equation}\label{e.log}
            \theta^{2} \vert \log \theta \vert = \varepsilon^{2}
    \end{equation}

    for small $\varepsilon$. Then
    \begin{align}
        & \theta_{1} (\varepsilon)^{2}  = 2 \varepsilon^{2} \vert \ln 2\varepsilon^2 \vert^{-1} + \text{o}\left( \varepsilon^{2} \vert \ln \varepsilon \vert^{-1} \right)\label{e.asymptotic1}
        \\
        & \theta_{2} (\varepsilon)= 1-\varepsilon^{2} + \text{o} \left( \varepsilon^{2} \right), & \theta_{3} (\varepsilon)= 1+\varepsilon^{2} + \text{o} \left( \varepsilon^{2} \right). \label{e.asymptotic2}
    \end{align}
\end{lem}
\begin{proof}
     It is then easy to see that 
    \begin{align*}
       &  \theta_{1}(\varepsilon) = \exp\left( \frac{1}{2} W_{-1} (-2\varepsilon^{2} )  \right), &\theta_{3} (\varepsilon) = \exp \left( \frac{1}{2} W_{0} (2\varepsilon^{2}) \right),
    \end{align*}
    and \eqref{e.asymptotic1} then follows from \eqref{eqn.prop.lambert}.
Let us recall that  $\theta_{2}$ is the solution to \eqref{e.log} such that $\theta_{2} <1$ and $\theta_{2} \rightarrow 1$ as $\varepsilon\rightarrow 0$. If we write $\theta_{2} (\varepsilon) = 1 - w(\varepsilon)$, and take derivatives of \eqref{e.log} it follows that 
\begin{align*}
    & w^{\prime} (\varepsilon) (1-w(\varepsilon)) \left( 1+2\log(1-w(\varepsilon))\right) = 2\varepsilon,
    \\
    & w^{\prime\prime} (\varepsilon)  (1-w(\varepsilon)) \left( 1+2\log(1-w(\varepsilon))\right) - (w^{\prime} (\varepsilon))^{2} \left( 3+ 2\log(1-w(\varepsilon))  \right)=2,
\end{align*}
and letting $\varepsilon\rightarrow 0$ we get that $w^{\prime}(0)=0$, $w^{\prime \prime}(0)=2$.    
\end{proof}

\subsection{Basic estimates}
For a proof of the following Lemmas, see e.g. \cite[Lemma B2, Lemma B3]{CarfagniniTodino25}.

\begin{lem}\label{lemma.taylor.expansion.Pl}
    Let $P_{\ell}$ be the $\ell$-th Legendre polynomial. Then there exists a constant $c$ such that for any $\theta$ small enough 
    \begin{equation}
        0< 1- P_{\ell} (\cos \theta )\leqslant c (\ell^{2} \theta^{2} + \ell^{4} \theta^{4} ).
    \end{equation}
\end{lem}

\begin{lem}\label{lemma.integral.bound}
   For any $p<1$ there exists an explicit constant $c_{p}$ such that for any $0<a<1$ 
   \begin{align}
       \int_{0}^{a} u^{-p} \sqrt{- \log u} du \leqslant c_{p} a^{1-p} \sqrt{-\log a}. \label{eqn.integral.bound}
   \end{align}
\end{lem}

\end{document}